\documentclass[11pt]{article}

\usepackage{amsmath}
\usepackage{graphicx}
\usepackage{epstopdf}
\usepackage{amssymb}
\usepackage{mathrsfs}
\usepackage{amsfonts}
\usepackage{mathrsfs,amscd,amssymb,amsthm,amsmath,bm,graphicx,psfrag}
\usepackage[numbers,sort&compress]{natbib}

\makeatletter
\renewcommand{\@seccntformat}[1]{{\csname the#1\endcsname}{\normalsize.}\hspace{.5em}}

\makeatother

\numberwithin{equation}{section}

\def \[{\begin{equation}}
\def \]{\end{equation}}

\newtheorem{thm}{Theorem}[section]
\newtheorem{defi}[thm]{Definition}
\newtheorem{lem}[thm]{Lemma}

\newtheorem{prop}[thm]{Proposition}

\newenvironment{kst}
{\setlength{\leftmargini}{1.3\parindent}
 \begin{itemize}
 \setlength{\itemsep}{-1.1mm}}
{\end{itemize}}

\begin{document}
\setlength{\baselineskip}{16.5pt}

\begin{center}{\Large \bf Tur\'{a}n problems for star-path forests in hypergraphs}

\vspace{4mm} {\large Junpeng Zhou$^{\rm{a,b}}$, Xiying Yuan$^{\rm{a,b}*}$} \vspace{2mm}

{\small $^\text{a}$ Department of Mathematics, Shanghai University, Shanghai 200444, P.R. China}

{\small $^\text{b}$ Newtouch Center for Mathematics of Shanghai University, Shanghai 200444, P.R. China}

\end{center}

\footnotetext{*Corresponding author. Email address: xiyingyuan@shu.edu.cn (Xiying Yuan)}

\footnotetext{junpengzhou@shu.edu.cn (Junpeng Zhou)}

\footnotetext{This work is supported by the National Natural Science Foundation of China (Nos.11871040, 12271337, 12371347)}

\begin{abstract}
An $r$-uniform hypergraph ($r$-graph for short) is linear if any two edges intersect at most one vertex.
Let $\mathcal{F}$ be a given family of $r$-graphs. An $r$-graph $H$ is called $\mathcal{F}$-free if $H$ does not contain any member of $\mathcal{F}$ as a subgraph. The Tur\'{a}n number of $\mathcal{F}$ is the maximum number of edges in any $\mathcal{F}$-free $r$-graph on $n$ vertices, and the linear Tur\'{a}n number of $\mathcal{F}$ is defined as the Tur\'{a}n number of $\mathcal{F}$ in linear host hypergraphs.
An $r$-uniform linear path $P^r_\ell$ of length $\ell$ is an $r$-graph with edges $e_1,\dots,e_\ell$ such that $|V(e_i)\cap V(e_j)|=1$ if $|i-j|=1$, and $V(e_i)\cap V(e_j)=\emptyset$ for $i\neq j$ otherwise. Gy\'{a}rf\'{a}s et al. [\textit{European J. Combin.} (2022) 103435] obtained an upper bound for the linear Tur\'{a}n number of $P_\ell^3$. In this paper, an upper bound for the linear Tur\'{a}n number of $P_\ell^r$ is obtained, which generalizes the known result of $P_\ell^3$ to any $P_\ell^r$. Furthermore, some results for the linear Tur\'{a}n number and Tur\'{a}n number of several linear star-path forests are obtained.
\end{abstract}



\section{\normalsize Introduction}
\ \ \ \
A \textit{hypergraph} $H=(V(H),E(H))$ consists of a vertex set $V(H)$ and an edge set $E(H)$, where each entry in $E(H)$ is a nonempty subset of $V(H)$. If $|V(e)|=r$ for every $e\in E(H)$, then $H$ is called an \textit{$r$-uniform hypergraph} ($r$-graph for short). A hypergraph $H$ is \textit{linear} if any two edges intersect at most one vertex. 
The \textit{degree} $d_H(v)$ of vertex $v$ (or simply $d(v)$) is defined as the number of edges containing $v$. Denote by $\Delta(H)$ and $\delta(H)$ the maximum degree and the minimum degree of $H$, respectively. For $U\subseteq V(H)$, let $H-U$ denote the $r$-graph obtained from $H$ by removing the vertices in $U$ and the edges incident to them.
Let $H\cup G$ denote the disjoint union of hypergraphs $H$ and $G$. The disjoint union of $k$ hypergraphs $H$ is denoted by $kH$.

Given a family $\mathcal{F}$ of $r$-graphs, an $r$-graph $H$ is called \textit{$\mathcal{F}$-free} if $H$ does not contain any member of $\mathcal{F}$ as a subgraph. The \textit{Tur\'{a}n number} ${\rm{ex}}_r(n,\mathcal{F})$ of $\mathcal{F}$ is the maximum number of edges in any $\mathcal{F}$-free $r$-graph on $n$ vertices. Collier-Cartaino et al. \cite{AA6} defined the linear Tur\'{a}n number of hypergraphs, and the \textit{linear Tur\'{a}n number} ${\rm{ex}}^{\rm{lin}}_r(n,\mathcal{F})$ of $\mathcal{F}$ is the Tur\'{a}n number of $\mathcal{F}$ in linear host hypergraphs. If $\mathcal{F}=\{G\}$, then we use ${\rm{ex}}_r(n,G)$ and ${\rm{ex}}^{\rm{lin}}_r(n,G)$ in place of ${\rm{ex}}_r(n,\{G\})$ and ${\rm{ex}}^{\rm{lin}}_r(n,\{G\})$.

The Tur\'{a}n number of hypergraphs has been studied extensively \cite{AA9,AA10,AA11}. In particular, the linear Tur\'{a}n number of hypergraphs have been studied in \cite{AA2,AA6,AA12,B5,B8,D4}.
An \textit{$r$-uniform linear cycle} $C^r_{\ell}$ of length $\ell$ is an $r$-graph with edges $e_1,\dots,e_\ell$ such that $|V(e_i)\cap V(e_{i+1})|=1$ for all $1\leq i\leq\ell-1$, $|V(e_\ell)\cap V(e_1)|=1$ and $V(e_i)\cap V(e_j)=\emptyset$ for all other pairs $\{i,j\},i\neq j$ \cite{AA6}.
Determining the ${\rm{ex}}^{\rm{lin}}_3(n,C_3^3)$ is equivalent to the famous $(6,3)$-problem. In 1976, Ruzsa and Szemer\'{e}di \cite{AA7} showed that
\begin{eqnarray*}
\frac{n^2}{e^{O(\sqrt{\log n})}}<{\rm{ex}}^{\rm{lin}}_3(n,C_3^3)=o(n^2).
\end{eqnarray*}
For integers $r,\ell\geq3$, Collier-Cartaino et al. \cite{AA6} proved that
\begin{eqnarray*}
{\rm{ex}}^{\rm{lin}}_r(n,C_\ell^r)\leq cn^{1+\frac{1}{\lfloor\frac{\ell}{2}\rfloor}},
\end{eqnarray*}
where $c=c(r,\ell)>0$ is some constant.

An \textit{$r$-uniform linear path} $P^r_\ell$ of length $\ell$ is an $r$-graph with edges $e_1,\dots,e_\ell$ such that $|V(e_i)\cap V(e_j)|=1$ if $|i-j|=1$, and $V(e_i)\cap V(e_j)=\emptyset$ for $i\neq j$ otherwise. In the case of the $r$-uniform linear path, Kostochka et al. \cite{AA3} proved that ${\rm{ex}}_r(n,P_\ell^r)={\rm{ex}}_r(n,C_\ell^r)$ for fixed $r\geq3$, $\ell\geq4$, $(\ell,r)\neq(4,3)$ and sufficiently large $n$.
For integers $r\geq3,\ell\geq1,k\geq2$ and sufficiently large $n$, Bushaw and Kettle \cite{AA5} determined the ${\rm{ex}}_r(n,kP_\ell^r)$.
For the linear Tur\'{a}n number of $P_\ell^r$, Gy\'{a}rf\'{a}s et al. \cite{AA2} showed that
\begin{eqnarray}
{\rm{ex}}_3^{\rm{lin}}(n,P_\ell^3)\leq 1.5\ell n
\end{eqnarray}
for any $\ell\geq3$ and mentioned that it does not seem easy to find the asymptotic of the linear Tur\'{a}n number of $P_\ell^r$. We obtain an upper bound for ${\rm{ex}}_r^{\rm{lin}}(n,P_\ell^r)$, which generalizes (1.1) to any $P_\ell^r$ as follows.

\begin{thm}
Fix integers $r\geq3$ and $\ell\geq4$. Then
\begin{eqnarray*}
{\rm{ex}}_r^{\rm{lin}}(n,P_\ell^r)\leq \frac{(2r-3)\ell}{2}n.
\end{eqnarray*}
\end{thm}

Let $S_\ell$ denote the star with $\ell$ edges. An \textit{$r$-uniform linear star} $S^r_\ell$ is an $r$-graph 
obtained from $S_\ell$ by inserting $r-2$ new distinct vertices in each edge of $S_\ell$. In the case of the $r$-uniform linear star, Duke and Erd\H{o}s \cite{AA4} gave an upper bound for the Tur\'{a}n number of $S_\ell^r$. For integers $\ell,k\geq1,r\geq2$ and sufficiently large $n$, Khormali and Palmer \cite{D4} determined the ${\rm{ex}}_r(n,kS_\ell^r)$ and ${\rm{ex}}_r^{\rm{lin}}(n,kS_\ell^r)$.

Based on some specific linear $r$-graphs, which are constructed by using the combinational design, integer lattice and Cartesian product of hypergraphs, we give some lower bounds for the linear Tur\'{a}n number of the linear path and a kind of linear star-path forests in Section 3.

Moreover, an upper bound for the linear Tur\'{a}n number of this kind of linear star-path forests will be presented in Theorem 1.2.

\begin{thm}
Fix integers $r\geq3$, $\ell_0\geq4$ and $k\geq0$. Let $\ell_1\geq\ell_2\geq\cdots\geq\ell_k$ and $\ell=\max\{\ell_0,\ell_1\}$ when $k\geq1$.
Then for sufficiently large $n$,
\begin{eqnarray*}
{\rm{ex}}_r^{\rm{lin}}(n,P_{\ell_0}^r\cup S_{\ell_1}^r\cup\cdots\cup S_{\ell_k}^r)\leq \Big(\frac{k}{r-1}+\frac{(2r-3)\ell}{2}\Big)(n-k)+\frac{k(k-1)}{r(r-1)}.
\end{eqnarray*}
\end{thm}

The following two bounds for the Tur\'{a}n number of two kinds of linear star-path forests are obtained in Section 5.

\begin{thm}
Fix integers $r\geq3$, $\ell\geq4$ and $k\geq0$. For sufficiently large $n$,
\begin{eqnarray*}
\binom{n}{r}-\binom{n-k}{r}+{\rm{ex}}_r(n-k,\{P_\ell^r, S_\ell^r\})\leq {\rm{ex}}_r(n,P_\ell^r\cup kS_\ell^r)\leq \binom{n}{r}-\binom{n-k}{r}+{\rm{ex}}_r(n-k,P_\ell^r).
\end{eqnarray*}
\end{thm}

\begin{thm}
Fix integers $r\geq3$, $\ell\geq4$, $k_1\geq2$ and $k_2\geq0$. For sufficiently large $n$,
\begin{eqnarray*}
{\rm{ex}}_r(n,k_1P_\ell^r\cup k_2S_\ell^r)\leq \binom{n}{r}-\binom{n-k_2}{r}+{\rm{ex}}_r(n-k_2,k_1P_\ell^r)
\end{eqnarray*}
and
\begin{eqnarray*}
{\rm{ex}}_r(n,k_1P_\ell^r\cup k_2S_\ell^r)\geq \binom{n}{r}-\binom{n-k_1-k_2+1}{r}+{\rm{ex}}_r(n-k_1-k_2+1,\{P_\ell^r, S_\ell^r\}).
\end{eqnarray*}
\end{thm}

\section{\normalsize Linear Tur\'{a}n number for linear path}
\ \ \ \
Let $\ell\geq3$ and $H$ be a linear $r$-graph containing a $P_{\ell-1}^r$ with edges $$e_i=\{v_{(i-1)(r-1)+1},\cdots,v_{i(r-1)},v_{i(r-1)+1}\},\ \ i=1,\cdots,\ell-1.$$
Inspired by \cite{AA2}, we call $v_1,\cdots,v_{r-1}$ the left end vertices of $P_{\ell-1}^r$, $v_{(\ell-2)(r-1)+2},\cdots,v_{(\ell-1)(r-1)+1}$ the right end vertices of $P_{\ell-1}^r$, the other vertices of $P_{\ell-1}^r$ the interior vertices of $P_{\ell-1}^r$, and the vertices in $V(H)\backslash V(P_{\ell-1}^r)$ the exterior vertices.

For any left end vertex $u$ of $P_{\ell-1}^r$, let $A_1(u)$ denote the set of edges in $H$ containing the left end vertex $u$, one interior vertex, and $r-2$ exterior vertices. Since $e_1$ contains the interior $v_r$ and all left end vertices,
we have $e_1\notin A_1(v_i)$. Note that the number of interior vertices of $P_{\ell-1}^r$ is $(r-1)(\ell-3)+1$. By the linearity of $H$, we have
\begin{eqnarray*}
|A_1(v_i)|\leq (r-1)(\ell-3)+1-1=(r-1)(\ell-3)
\end{eqnarray*}
for $1\leq i\leq r-1$. Set $A_1=\bigcup_{i=1}^{r-1}A_1(v_i)$. Then
\begin{eqnarray*}
|A_1|=\sum_{i=1}^{r-1}|A_1(v_i)|\leq (r-1)^2(\ell-3)
\end{eqnarray*}
with the equality holds if and only if for any $1\leq i\leq r-1$, $v_i$ and every interior vertex except $v_r$ are contained by some edge of $A_1$.

Similarly, for any right end vertex $v$, let $B_1(v)$ denote the set of edges in $H$ containing the right end vertex $v$, one interior vertex, and $r-2$ exterior vertices. Set $B_1=\bigcup_{j=(\ell-2)(r-1)+2}^{(\ell-1)(r-1)+1}B_1(v_j)$. Then
\begin{eqnarray*}
|B_1|=\sum_{j=(\ell-2)(r-1)+2}^{(\ell-1)(r-1)+1}|B_1(v_j)|\leq (r-1)^2(\ell-3)
\end{eqnarray*}
with the equality holds if and only if for any $(\ell-2)(r-1)+2\leq j\leq (\ell-1)(r-1)+1$, $v_j$ and every interior vertex except $v_{(\ell-2)(r-1)+1}$ are contained by some edge of $B_1$.

Since $A_1\cap B_1=\emptyset$, we have
\begin{eqnarray}
|A_1\cup B_1|=|A_1|+|B_1|=\sum_{i=1}^{r-1}|A_1(v_i)|+\sum_{j=(\ell-2)(r-1)+2}^{(\ell-1)(r-1)+1}|B_1(v_j)|\leq 2(r-1)^2(\ell-3).
\end{eqnarray}

For $f_1\in A_1$ and $f_2\in B_1$, if $v_{i+1}\in f_1$ and $v_{i}\in f_2$, then we call that the edges $f_1$ and $f_2$ traverse the interior vertices $v_i$ and $v_{i+1}$.

In the rest of this section, we always assume that $H$ is a $P_\ell^r$-free linear $r$-graph and $H$ contains a $P_{\ell-1}^r$. Label the vertices and edges of $P_{\ell-1}^r$ as $e_i=\{v_{(i-1)(r-1)+1},\cdots,v_{i(r-1)},v_{i(r-1)+1}\}$ for $i=1,\cdots,\ell-1$.

\begin{lem}
For $\ell\geq4$ and $2\leq i\leq \ell-2$, the following two cases does not occur in $H$.
\begin{kst}
\item[(\romannumeral1)] There exists some $(i-1)(r-1)+2\leq j\leq i(r-1)$ such that $v_j\in f_1\in A_1$ and $v_j\in f_2\in B_1$.
\item[(\romannumeral2)] There exist $f_1\in A_1$ and $f_2\in B_1$ such that $f_1$ and $f_2$ traverse the interior vertices $v_{i(r-1)},v_{i(r-1)+1}$ and $V(f_1)\cap V(f_2)=\emptyset$.
\end{kst}
\end{lem}
\begin{proof}[{\bf{Proof}}]
$(\romannumeral1)$. If $v_j\in f_1\in A_1$ and $v_j\in f_2\in B_1$ for some $(i-1)(r-1)+2\leq j\leq i(r-1)$, then by the linearity of $H$, we have $(V(f_1)\backslash \{v_j\})\cap (V(f_2)\backslash \{v_j\})=\emptyset$. Then the edges $e_{i+1},\cdots,e_{\ell-1}$, $f_2$, $f_1$, $e_1,\cdots,e_{i-1}$ form a $P_{\ell}^r$ in $H$, which contradicts
to the assumption that $H$ is $P_{\ell}^r$-free.

$(\romannumeral2)$. If $f_1\in A_1$ and $f_2\in B_1$ traverse the interior vertices $v_{i(r-1)},v_{i(r-1)+1}$ and $V(f_1)\cap V(f_2)=\emptyset$, then the edges $f_2$, $e_{\ell-1},\cdots,e_{i+1}$, $f_1$, $e_1,\cdots,e_{i-1}$ form a $P_\ell^r$ in $H$, which is a contradiction.
\end{proof}

\begin{lem} For $\ell\geq4$ and $2\leq i\leq \ell-2$, suppose that $f_1\in A_1(v_j)$ and $f_2\in B_1(v_k)$ traverse the interior vertices $v_{i(r-1)}$ and $v_{i(r-1)+1}$. Then we have the following results.
\begin{kst}
\item[(\romannumeral1)] There exists a left end vertex $u\neq v_j$ such that $\{u,v_{i(r-1)+1}\}$ is not contained by any edge of $A_1$, and there exists a right end vertex $w\neq v_k$ such that $\{w,v_{i(r-1)}\}$ is not contained by any edge of $B_1$.
\item[(\romannumeral2)] If $r\geq4$, then there exists a right end vertex $w'$ such that $\{w',v_t\}$ is not contained by any edge of $B_1$ for any $(i-1)(r-1)+2\leq t\leq (i-1)(r-1)+r-2$.
\end{kst}
\end{lem}
\begin{proof}[{\bf{Proof}}]
$(\romannumeral1)$. By Lemma 2.1$(\romannumeral2)$, we have $V(f_1)\cap V(f_2)\neq \emptyset$. Set $V(f_1)\cap V(f_2)=\{v'\}$, then $\{v_j,v_{i(r-1)+1},v'\}\subseteq V(f_1)$ and $\{v_k,v_{i(r-1)},v'\}\subseteq V(f_2)$. Suppose to the contrary that for each $1\leq s\leq r-1$ and $s\neq j$, $\{v_s,v_{i(r-1)+1}\}$ is contained by some edge of $A_1$. Denote these edges by $g_1,\cdots,g_{r-2}$. By the linearity of $H$, $V(g_1)\backslash\{v_{i(r-1)+1}\}$, $\cdots,V(g_{r-2})\backslash\{v_{i(r-1)+1}\}$ are pairwise disjoint.
Since $|V(f_2)\backslash\{v_k,v_{i(r-1)},v'\}|=r-3$, by pigeonhole principle, there exists some $1\leq t\leq r-2$ such that $\big(V(g_t)\backslash\{v_{i(r-1)+1}\}\big)\cap \big(V(f_2)\backslash\{v_k,v_{i(r-1)},v'\}\big)=\emptyset$. Then $V(g_t)\cap V(f_2)=\emptyset$. Note that $g_t$ and $f_2$ traverse the interior vertices $v_{i(r-1)},v_{i(r-1)+1}$, which contradicts to Lemma 2.1$(\romannumeral2)$.

Similarly, we may prove that there exists a right end vertex $w\neq v_k$ such that $\{w,v_{i(r-1)}\}$ is not contained by any edge of $B_1$.

$(\romannumeral2)$. Let $r\geq4$ and $(i-1)(r-1)+2\leq t\leq (i-1)(r-1)+r-2$.  Suppose to the contrary that for each $(\ell-2)(r-1)+2\leq s\leq (\ell-1)(r-1)+1$, $\{v_s,v_t\}$ is contained by some edge of $B_1$. Denote these edges by $h_1,\cdots,h_{r-1}$. By the linearity of $H$, $V(h_1)\backslash\{v_t\},\cdots,V(h_{r-1})\backslash\{v_t\}$ are pairwise disjoint. Since $|V(f_1)\backslash\{v_j,v_{i(r-1)+1}\}|=r-2$, by pigeonhole principle, there exists some $1\leq z\leq r-1$ such that $\big(V(h_z)\backslash\{v_t\}\big)\cap \big(V(f_1)\backslash\{v_j,v_{i(r-1)+1}\}\big)=\emptyset$. Then $V(h_z)\cap V(f_1)=\emptyset$. So the edges $h_z$, $e_{\ell-1},\cdots,e_{i+1}$, $f_1$, $e_1,\cdots,e_{i-1}$ form a $P_\ell^r$ in $H$, which is a contradiction.
\end{proof}

\begin{lem} For $\ell\geq4$, there exists a left end vertex $u$ and a right end vertex $v$ of $P_{\ell-1}^r$ such that $|A_1(u)|+|B_1(v)|\leq 2(r-2)(\ell-3)$.
\end{lem}
\begin{proof}[{\bf{Proof}}]
Label the vertices and edges of $P_{\ell-1}^r$ in $H$ as $e_i=\{v_{(i-1)(r-1)+1},\cdots,v_{i(r-1)},v_{i(r-1)+1}\}$ for $i=1,\cdots,\ell-1$. For the upper bound on $|A_1|+|B_1|$ in (2.1), all possible edges are counted. In fact, some of these edges do not exist, which will be called ``excluded edges''. We will distinguish into two cases to prove that for any fixed $2\leq i\leq \ell-2$, there are $2(r-1)$ excluded edges. 

{\bf{Case 1.}} $v_{i(r-1)}\notin f$ for any $f\in A_1\cup B_1$.

By the definition of $A_1$ and $B_1$, $\{v_j,v_{i(r-1)}\}$ is not contained by any edge of $A_1\cup B_1$ for $j=1,\cdots,r-1,(\ell-2)(r-1)+2,\cdots,(\ell-1)(r-1)+1$. Thus, we have $2(r-1)$ excluded edges in this case.

{\bf{Case 2.}} There exists an edge $f\in A_1\cup B_1$ such that $v_{i(r-1)}\in f$.

Set $v'=v_{(i-1)(r-1)+r-2}$ if $f\in A_1$ and $v'=v_{i(r-1)+1}$ if $f\in B_1$. We consider two cases as follows.

{\bf{Subcase 2.1.}} There is no edge $f'\in A_1\cup B_1$ such that $f$ and $f'$ traverse the interior vertices $v_{i(r-1)}$ and $v'$.

If $f\in A_1$, from the assumption of the subcase, we know that $\{v_j,v'\}$ is not contained by any edge of $B_1$ for any $(\ell-2)(r-1)+2\leq j\leq (\ell-1)(r-1)+1$. Since $v_{i(r-1)}\in f$, by Lemma 2.1$(\romannumeral1)$, $\{v_j,v_{i(r-1)}\}$ is not contained by any edge of $B_1$ for any $(\ell-2)(r-1)+2\leq j\leq (\ell-1)(r-1)+1$. Note that these edges are different. Thus, we have $2(r-1)$ excluded edges.

Similarly, if $f\in B_1$, we also have $2(r-1)$ excluded edges.

{\bf{Subcase 2.2.}} There exists an edge $f'\in A_1\cup B_1$ such that $f$ and $f'$ traverse the interior vertices $v_{i(r-1)}$ and $v'$.

If $f\in A_1$, then $v'=v_{(i-1)(r-1)+r-2}\in f'\in B_1$. Since $v_{i(r-1)}\in f$, by Lemma 2.1$(\romannumeral1)$, $\{v_k,v_{i(r-1)}\}$ is not contained by any edge of $B_1$ for any $(\ell-2)(r-1)+2\leq k\leq (\ell-1)(r-1)+1$. So we have $r-1$ excluded edges.
On the other hand, if $r\geq 4$, then $v'$ is exactly a vertex of degree 1 in the linear path. By Lemma 2.1$(\romannumeral1)$, $\{v_j,v'\}$ is not contained by any edge of $A_1$ for any $1\leq j\leq r-1$. So we have $r-1$ excluded edges. If $r=3$, then $v'$ is exactly a vertex of degree 2 in the linear path. 
By reversing the roles of the left and right end vertices, we may obtain a result similar to Lemma 2.2$(\romannumeral1)$, namely, we have $2$ excluded edges. Note that $A_1\cap B_1=\emptyset$. Thus, we have $2(r-1)$ excluded edges.

If $f\in B_1$, then $v'=v_{i(r-1)+1}\in f'$. By Lemma 2.1$(\romannumeral1)$, $\{v_s,v_{i(r-1)}\}$ is not contained by any edge of $A_1$ for any $1\leq s\leq r-1$. Moreover, we have at least $2$ excluded edges by Lemma 2.2$(\romannumeral1)$, and have at least $r-3$ excluded edges by Lemma 2.2$(\romannumeral2)$. Note that these edges are different. Thus, we have at least $2(r-1)$ excluded edges.

In total, there are at least $2(r-1)(\ell-3)$ excluded edges.
Then we get
\begin{eqnarray*}
|A_1|+|B_1|&=& \sum_{i=1}^{r-1}|A_1(v_i)|+\sum_{j=(\ell-2)(r-1)+2}^{(\ell-1)(r-1)+1}|B_1(v_j)| \\
&=& \Big(|A_1(v_1)|+|B_1(v_{(\ell-2)(r-1)+2})|\Big)+\cdots+\Big(|A_1(v_{r-1})|+|B_1(v_{(\ell-1)(r-1)+1})|\Big) \\
&\leq& 2(r-1)^2(\ell-3)-2(r-1)(\ell-3) \\
&=& 2(r-1)(r-2)(\ell-3).
\end{eqnarray*}
By pigeonhole principle, there exists some $1\leq i\leq r-1$ such that $|A_1(v_i)|+|B_1(v_{(\ell-2)(r-1)+i+1})|\leq 2(r-2)(\ell-3)$. This completes the proof.
\end{proof}

\begin{proof}[{\bf{Proof of Theorem 1.1.}}]
Suppose to the contrary that $H$ is a $P_\ell^r$-free linear $r$-graph on $n$ vertices with more than $\frac{(2r-3)\ell}{2}n$ edges. We may assume that $n$ and $\ell$ are minimal.
If there exists a vertex $u$ such that $d_H(u)\leq\frac{(2r-3)\ell}{2}$, then $|E(H-\{u\})|>\frac{(2r-3)\ell}{2}(n-1)$. Since $H-\{u\}$ is $P_\ell^r$-free on $n-1$ vertices, $n$ is not minimal, which is a contradiction. So $\delta(H)>\frac{(2r-3)\ell}{2}$.

By the minimality of $\ell$, $H$ contains a $P_{\ell-1}^r$. Label the vertices and edges of $P_{\ell-1}^r$ in $H$ as $e_i=\{v_{(i-1)(r-1)+1},\cdots,v_{i(r-1)},v_{i(r-1)+1}\}$ for $i=1,\cdots,\ell-1$.
By Lemma 2.3, without loss of generality, we may set
\begin{eqnarray}
|A_1(v_1)|+|B_1(v_{(\ell-1)(r-1)+1})|\leq 2(r-2)(\ell-3).
\end{eqnarray}

For $2\leq k\leq r-1$, denote by $A_k(u)$ the set of edges in $H$ containing the left end vertex $u$ of $P_{\ell-1}^r$ and intersecting $V(P_{\ell-1}^r)\backslash\{u\}$ in $k$ vertices. Since $H$ is linear, each vertex in $V(P_{\ell-1}^r)\backslash\{u\}$ belongs to at most one edge in $\bigcup_{k=1}^{r-1}A_k(u)$. Note that $|V(P_{\ell-1}^r)|=(r-1)(\ell-1)+1$. Then
\begin{eqnarray}
\sum_{k=1}^{r-1}k|A_k(v_1)|\leq (r-1)(\ell-1).
\end{eqnarray}

Similarly, for $2\leq k\leq r-1$, denote by $B_k(v)$ the set of edges in $H$ containing the right end vertex $v$ and intersecting $V(P_{\ell-1}^r)\backslash\{v\}$ in $k$ vertices. Then we have  
\begin{eqnarray}
\sum_{k=1}^{r-1}k|B_k(v_{(\ell-1)(r-1)+1})|\leq (r-1)(\ell-1).
\end{eqnarray}

Combining (2.2)-(2.4), we get
\begin{eqnarray*}
&{}&|A_1(v_1)|+\sum_{k=1}^{r-1}k|A_k(v_1)|+|B_1(v_{(\ell-1)(r-1)+1})|+\sum_{k=1}^{r-1}k|B_k(v_{(\ell-1)(r-1)+1})| \\
&=& 2|A_1(v_1)|+\sum_{k=2}^{r-1}k|A_k(v_1)|+2|B_1(v_{(\ell-1)(r-1)+1})|+\sum_{k=2}^{r-1}k|B_k(v_{(\ell-1)(r-1)+1})| \\
&\leq& 2(r-2)(\ell-3)+2(r-1)(\ell-1),
\end{eqnarray*}
that is
\begin{eqnarray*}
|A_1(v_1)|+\sum_{k=2}^{r-1}\frac{k}{2}|A_k(v_1)|+|B_1(v_{(\ell-1)(r-1)+1})|+\sum_{k=2}^{r-1}\frac{k}{2}|B_k(v_{(\ell-1)(r-1)+1})| \\
\leq (r-2)(\ell-3)+(r-1)(\ell-1).
\end{eqnarray*}
By pigeonhole principle, then either
\begin{eqnarray*}
|A_1(v_1)|+\sum_{k=2}^{r-1}\frac{k}{2}|A_k(v_1)|\leq \frac{(r-2)(\ell-3)+(r-1)(\ell-1)}{2}
\end{eqnarray*}
or
\begin{eqnarray*}
|B_1(v_{(\ell-1)(r-1)+1})|+\sum_{k=2}^{r-1}\frac{k}{2}|B_k(v_{(\ell-1)(r-1)+1})|\leq \frac{(r-2)(\ell-3)+(r-1)(\ell-1)}{2}.
\end{eqnarray*}

Now let us consider the degree of $v_1$ and $v_{(\ell-1)(r-1)+1}$ in $H$. Note that $A_1(v_1)$ is the set of edges in $H$ containing $v_1$, one interior vertex and $r-2$ exterior vertices, and the number of edges in $H$ containing $v_1$, one right end vertex and $r-2$ exterior vertices is at most $r-1$.
Thus, the number of edges in $H$ containing $v_1$ and intersecting $V(P_{\ell-1}^r)\backslash\{v_1\}$ in one vertex is at most $|A_1(v_1)|+r-1$. Furthermore, there is no edge $f$ in $H$ such that $V(f)\cap V(P_{\ell-1}^r)=\{v_1\}$. Otherwise, the edges $f$, $e_1,\cdots,e_{\ell-1}$ form a $P_\ell^r$ in $H$, which is a contradiction. Thus, $d_H(v_1)\leq \sum_{k=1}^{r-1}|A_k(v_1)|+r-1$. Similarly, we have $d_H(v_{(\ell-1)(r-1)+1})\leq \sum_{k=1}^{r-1}|B_k(v_{(\ell-1)(r-1)+1})|+r-1$.
Thus, we have either
\begin{eqnarray*}
d_H(v_1)&\leq& \sum_{k=1}^{r-1}|A_k(v_1)|+r-1 \\
&\leq& |A_1(v_1)|+\sum_{k=2}^{r-1}\frac{k}{2}|A_k(v_1)|+r-1 \\
&\leq& \frac{(r-2)(\ell-3)+(r-1)(\ell-1)}{2}+r-1 \\
&<& \frac{(2r-3)\ell}{2}
\end{eqnarray*}
or
\begin{eqnarray*}
d_H(v_{(\ell-1)(r-1)+1})&\leq& \sum_{k=1}^{r-1}|B_k(v_{(\ell-1)(r-1)+1})|+r-1 \\
&\leq& |B_1(v_{(\ell-1)(r-1)+1})|+\sum_{k=2}^{r-1}\frac{k}{2}|B_k(v_{(\ell-1)(r-1)+1})|+r-1 \\
&\leq& \frac{(r-2)(\ell-3)+(r-1)(\ell-1)}{2}+r-1 \\
&<& \frac{(2r-3)\ell}{2},
\end{eqnarray*}
which is a contradiction to the fact $\delta(H)>\frac{(2r-3)\ell}{2}$. This completes the proof.
\end{proof}

Note that $\ell\geq4$ in Theorem 1.1. For $\ell=2$, it is easy to see that ${\rm{ex}}_r^{\rm{lin}}(n,P_2^r)=\lfloor\frac{n}{r}\rfloor$. For $\ell=3$, we have the following result.
\begin{lem}
For $r\geq3$, we have ${\rm{ex}}_r^{\rm{lin}}(n,P_3^r)\leq n$.
\end{lem}
\begin{proof}[{\bf{Proof.}}]
Let $H$ be a $P_3^r$-free linear $r$-graph on $n$ vertices with more than $n$ edges. We may assume that $n$ is minimal.
If there exists a vertex $u$ such that $d_H(u)<2$, then $|E(H-\{u\})|>n-1$. Since $H-\{u\}$ is $P_3^r$-free on $n-1$ vertices, $n$ is not minimal, which is a contradiction. So $\delta(H)\geq2$.
If $\Delta(H)\leq r$, then $\sum_{v\in V(H)}d_H(v)\leq nr$. Thus $|E(H)|\leq n$, which is a contradiction. So there exists a vertex $w$ such that $d_H(w)=\Delta(H)\geq r+1$.
Take a star $S_{r+1}^r$ with center $w$ and a vertex $w'\in V(S_{r+1}^r)\backslash\{w\}$. Since $d_H(w')\geq2$ and $H$ is linear, there is an edge $e$ such that $w'\in e$ and $w\notin e$. Note that there exists an edge $e'\in E(S_{r+1}^r)$ such that $V(e')\cap V(e)=\emptyset$. Then $e$, the edge containing the vertices $w,w'$ and $e'$ form a $P_3^r$, which is a contradiction.
\end{proof}

\section{\normalsize Construction of some linear $r$-graphs}
\ \ \ \
In this section, we introduce two special classes of linear $r$-graphs, one is a combinational design and the other is the integer lattice.
By some specific linear $r$-graphs, which constructed based on these two classes of linear $r$-graphs and Cartesian product of hypergraphs, we obtain a lower bound for the linear Tur\'{a}n number of the linear path and a kind of linear star-path forests.

\begin{defi}[\cite{AA1}]
A pairwise balanced design of index $\lambda$ (in brief, a $\lambda$-PBD) is a pair $(X,\mathcal{A})$ where $X$ is a set of points, $\mathcal{A}$ is a family of subsets of $X$ (called blocks), each $E\in \mathcal{A}$ containing at least two points, and such that for any pair $x,y$ of distinct points, there are exactly $\lambda$ blocks containing both.
\end{defi}

\begin{defi}[\cite{AA1}]
A $\lambda$-PBD on $n$ points in which all blocks have the same cardinality $r$ is traditionally called a $(n,r,\lambda)$-BIBD (balanced incomplete block design), or a $2$-$(n,r,\lambda)$ design.
\end{defi}

Note that a 2-$(n,r,\lambda)$ design corresponds to a $r$-graph on $n$ vertices, where a block corresponds to an edge of the $r$-graph. In particular, a 2-$(n,r,1)$ design corresponds to a linear $r$-graph on $n$ vertices in which any two vertices are contained in exactly one edge. Throughout this paper, we will not distinguish the 2-$(n,r,1)$ design and the linear $r$-graph corresponding to the design. Denote the 2-$(n,r,1)$ design by $D_{n,r}$. The Fano plane is a $D_{7,3}$. Theorem 3.3 is an existence theory for $D_{n,r}$.

\begin{thm}[\cite{AA1}]
Given positive integer $r$, the $D_{n,r}$ exists for all sufficiently large integers $n$ whose congruence
\begin{eqnarray}
(n-1)\equiv0 ({\rm{mod}}\ r-1)\ \ \ {\rm{and}}\ \ \ n(n-1)\equiv0 ({\rm{mod}}\ r(r-1))
\end{eqnarray}
holds.
\end{thm}

\begin{prop}[\cite{AA1}]
The quotients $\frac{n-1}{r-1}$ and $\frac{n(n-1)}{r(r-1)}$ are, respectively, the number of blocks on a point and the total number of blocks in $D_{n,r}$.
\end{prop}

\begin{thm}
Let integers $c>0$, $r\geq2$, $\ell\geq2$ and $\ell,r$ satisfy $(\ell-1)\big((\ell-1)(r-1)+1\big)\equiv0 ({\rm{mod}}\ r)$. If $\ell$ is large and $n=c\big((\ell-1)(r-1)+1\big)$, then
\begin{eqnarray*}
{\rm{ex}}_r^{\rm{lin}}(n,P_\ell^r)\geq \frac{(\ell-1)n}{r}.
\end{eqnarray*}
\end{thm}
\begin{proof}[{\bf{Proof}}]
Since $\ell$ is large and $(\ell-1)(r-1)+1,r$ satisfy (3.1), by Theorem 3.3, the linear $r$-graph $D_{(\ell-1)(r-1)+1,r}$ exists.
Let $\widetilde{H}=cD_{(\ell-1)(r-1)+1,r}$. Since $|V(D_{(\ell-1)(r-1)+1,r})|<|V(P_\ell^r)|=\ell(r-1)+1$, $D_{(\ell-1)(r-1)+1,r}$ is $P_\ell^r$-free. Thus $\widetilde{H}$ is $P_\ell^r$-free.
By Proposition 3.4, we have
\begin{eqnarray*}
|E(\widetilde{H})|=c\cdot \frac{\big((\ell-1)(r-1)+1\big)(\ell-1)(r-1)}{r(r-1)}=\frac{(\ell-1)n}{r}.
\end{eqnarray*}
\end{proof}

Now we introduce a $d$-regular linear $r$-graph on $r^d$ vertices formed by $d$-tuples from $\{1,2,\cdots,r\}$: the collection of $d$-tuples that are fixed in all but one coordinate forms an edge. This linear $r$-graph is called integer lattice and is denoted as $[r]^d$ in \cite{D4}. Note that $[r]^d$ has $d\cdot r^{d-1}$ edges, and
these edges may be partitioned into $d$ classes each of which forms a matching of size $r^{d-1}$. This also gives a proper edge-coloring of $[r]^d$. In Figure 1 we show an example of $[4]^3$, where each edge of $[4]^3$ is represented by a line segment.

\begin{figure}[h]
  \centering
  \includegraphics[scale=1.2]{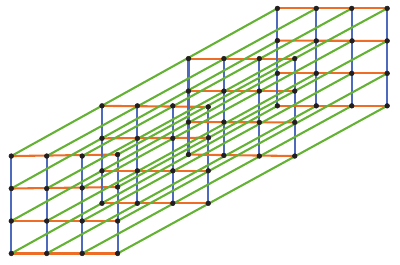}\\
  \caption{\small{Hypergraph $[4]^3$.}}\label{fig2}
\end{figure}

Khormali and Palmer \cite{D4} introduced the Cartesian product of hypergraphs when studying the Tur\'{a}n problem.
For two hypergraphs $H$ and $G$, the Cartesian product $H\square G$ of $H$ and $G$ is the hypergraph with the vertex set $V(H)\times V(G)$ and the edge set
$$E(H\square G)=\big\{e\times\{u\}\,|\,e\in E(H),u\in V(G)\big\}\cup \big\{\{v\}\times e'\,|\,e'\in E(G),v\in V(H)\big\}.$$
By the definition of $H\square G$, we have $|V(H\square G)|=|V(H)|\cdot|V(G)|$ and $|E(H\square G)|=|E(H)|\cdot|V(G)|+|E(G)|\cdot|V(H)|$.

Note that if $H$ is an $r_1$-graph and $G$ is an $r_2$-graph, then the edges in $\big\{e\times\{u\}\,|\,e\in E(H),u\in V(G)\big\}$ are of order $r_1$ and the edges in $\big\{\{v\}\times e'\,|\,e'\in E(G),v\in V(H)\big\}$ are of order $r_2$. Moreover, if $H$ and $G$ are linear, then $H\square G$ is linear.

Now we will construct a linear $r$-graph based on the Cartesian product of a $D_{k,r}$ and an integer lattice, which will imply an lower bound on ${\rm{ex}}_r^{\rm{lin}}(n,P_{\ell_0}^r\cup S_{\ell_1}^r\cup\cdots\cup S_{\ell_k}^r)$.

\begin{thm}
Suppose integers $r\geq3$, $\ell_0\geq4$, $k\geq1$ and $c>0$. Let $\ell_1\geq\ell_2\geq\cdots\geq\ell_k$, $\ell=\min\{\ell_0,\ell_k\}$ and $n=c(r-1)^k\big((\ell-1)(r-1)+1\big)+k$. If $\ell,r$ satisfy $(\ell-1)\big((\ell-1)(r-1)+1\big)\equiv0 ({\rm{mod}}\ r)$, and $k,r$ satisfy $(k-1)\equiv0 ({\rm{mod}}\ r-1)$ and $k(k-1)\equiv0 ({\rm{mod}}\ r(r-1))$, then for large $\ell$ and $k$,
\begin{eqnarray*}
{\rm{ex}}_r^{\rm{lin}}(n,P_{\ell_0}^r\cup S_{\ell_1}^r\cup\cdots\cup S_{\ell_k}^r)\geq \Big(\frac{k}{r-1}+\frac{\ell-1}{r}\Big)(n-k)+\frac{k(k-1)}{r(r-1)}.
\end{eqnarray*}
\end{thm}
\begin{proof}[{\bf{Proof}}]
Since $\ell,k$ are large, and $(\ell-1)(r-1)+1,r$ and $k,r$ satisfy (3.1), by Theorem 3.3, the linear $r$-graphs $D_{(\ell-1)(r-1)+1,r}$ and $D_{k,r}$ exist.

Note that the edges of order $r-1$ in $E(D_{(\ell-1)(r-1)+1,r}\square[r-1]^k)$ can be partitioned into $k$ classes with different colors, denoted these colors by $c_1,\cdots,c_{k}$, where there are $(r-1)^{k-1}((\ell-1)(r-1)+1)$ edges for each color.
Let $\widehat{H}$ be a linear $r$-graphs on $n$ vertices obtained from $D_{k,r}\cup c\big(D_{(\ell-1)(r-1)+1,r}\square[r-1]^k\big)$ as follows:
for each $D_{(\ell-1)(r-1)+1,r}\square[r-1]^k$ and $1\leq i\leq k$, inserting $v_i\in V(D_{k,r})$ to each of $(r-1)^{k-1}((\ell-1)(r-1)+1)$ edges of color $c_i$ of $E(D_{(\ell-1)(r-1)+1,r}\square[r-1]^k)$.

Since $|V(D_{(\ell-1)(r-1)+1,r})|<|V(P_\ell^r)|=|V(S_\ell^r)|=\ell(r-1)+1$, $D_{(\ell-1)(r-1)+1,r}$ is $P_\ell^r$-free and $S_\ell^r$-free, and $D_{(\ell-1)(r-1)+1,r}\square[r-1]^k$ is $P_\ell^r$-free and $S_\ell^r$-free. Then any $P_{\ell_0}^r$ or $S_{\ell_i}^r$ in $\widehat{H}$ contains at least one vertex in $V(D_{k,r})$, where $1\leq i\leq k$.
So $P_{\ell_0}^r\cup S_{\ell_1}^r\cup\cdots\cup S_{\ell_k}^r$ in $\widehat{H}$ contains at least $k+1$ distinct vertices in $V(D_{k,r})$. Since $|V(D_{k,r})|=k$, $\widehat{H}$ is $P_{\ell_0}^r\cup S_{\ell_1}^r\cup\cdots\cup S_{\ell_k}^r$-free.

By Proposition 3.4 and the definition of Cartesian product, we have
\begin{eqnarray*}
|E(\widehat{H})|&=& c\Big(\big((\ell-1)(r-1)+1\big)\cdot k(r-1)^{k-1}+(r-1)^k\cdot\frac{\big((\ell-1)(r-1)+1\big)(\ell-1)}{r}\Big)+\frac{k(k-1)}{r(r-1)} \\
&=& \Big(\frac{k}{r-1}+\frac{\ell-1}{r}\Big)(n-k)+\frac{k(k-1)}{r(r-1)}.
\end{eqnarray*}
\end{proof}

\section{\normalsize Linear Tur\'{a}n number for a kind of linear star-path forests}
\ \ \ \
In this section, we mainly present the proof of Theorem 1.2, which states an upper bound for the linear Tur\'{a}n number of $P_{\ell_0}^r\cup S_{\ell_1}^r\cup\cdots\cup S_{\ell_k}^r$. We first introduce a kind of $r$-graph and give an upper bound on its size.

Let $\mathcal{H}$ be a family of linear $r$-graphs. For any $H\in \mathcal{H}$, $H$ has a vertex set $V(H)=U\cup\bar{U}$ and an edge set $E(H)$, where for any edge $e\in E(H)$, $e$ satisfies $V(e)\cap U\neq\emptyset$ and $V(e)\cap \bar{U}\neq\emptyset$.

Let integer $c>0$, $n=c(r-1)^k+k$ and the set of vertices $U=\{v_1,\cdots,v_k\}$. Note that the edges in $E([r-1]^k)$ can be partitioned into $k$ classes with different colors, denoted these colors by $c_1,\cdots,c_{k}$, where there are $(r-1)^{k-1}$ edges for each color.
Let $H^*$ be a linear $r$-graph on $n$ vertices obtained from $U\cup \frac{n-k}{(r-1)^k}[r-1]^k$ as follows:
for each $[r-1]^k$ and $1\leq i\leq k$, inserting $v_i\in U$ to each of $(r-1)^{k-1}$ edges of color $c_i$ of $E([r-1]^k)$. Then $H^*\in \mathcal{H}$.

By the definition of the integer lattice $[r-1]^k$, we have
\begin{eqnarray*}
|E(H^*)|= \frac{n-k}{(r-1)^k}\cdot k(r-1)^{k-1}=\frac{k(n-k)}{r-1}.
\end{eqnarray*}

Inspired by the proof of Theorem 3 in \cite{D4}, we have the following result.

\begin{lem}
Let $H\in \mathcal{H}$ be a linear $r$-graph on $n$ vertices. Then
\begin{eqnarray*}
|E(H)|\leq \frac{|U|(n-|U|)}{r-1}.
\end{eqnarray*}
Moreover, we have $|E(H^*)|=\frac{|U|(n-|U|)}{r-1}$.
\end{lem}
\begin{proof}[{\bf{Proof}}]
Let $|U|=k$. Then $|\bar{U}|=n-k$. To estimate $|E(H)|$, we consider the number of pairs $(e,\{u,v\})$. For $e\in E(H)$, we have $|V(e)\cap U|$ choices for $u$ and $|V(e)\cap \bar{U}|$ choices for $v$. Note that $|V(e)\cap U|+|V(e)\cap \bar{U}|=r$. Then we have
\begin{eqnarray*}
(r-1)|E(H)|\leq \sum_{e\in E(H)}|V(e)\cap U|\cdot|V(e)\cap \bar{U}|.
\end{eqnarray*}
Moreover, for any given $u$ and $v$, there is at most one edge containing $u,v$ by the linearity of $H$. So the number of pairs $(e,\{u,v\})$ is at most $|U|\cdot|\bar{U}|$, i.e.
\begin{eqnarray*}
\sum_{e\in E(H)}|V(e)\cap U|\cdot|V(e)\cap \bar{U}|\leq k(n-k).
\end{eqnarray*}
Thus, we obtain $|E(H)|\leq \frac{k(n-k)}{r-1}$.
\end{proof}

\begin{thm}[\cite{D4}]
Fix integers $\ell,k\geq1$ and $r\geq2$. Then for sufficiently large $n$,
\begin{eqnarray*}
{\rm{ex}}_r^{\rm{lin}}(n,kS_\ell^r)\leq \Big(\frac{\ell-1}{r}+\frac{k-1}{r-1}\Big)(n-k+1)+\frac{\binom{k-1}{2}}{\binom{r}{2}}.
\end{eqnarray*}
Furthermore, this bound is sharp asymptotically.
\end{thm}

\begin{proof}[{\bf{Proof of Theorem 1.2.}}]
In the star-path forest $P_{\ell_0}^r\cup S_{\ell_1}^r\cup\cdots\cup S_{\ell_k}^r$, there exists $k$ stars.
If $k=0$, then by Theorem 1.1, we have ${\rm{ex}}_r^{\rm{lin}}(n,P_{\ell_0}^r)\leq \frac{(2r-3){\ell_0}}{2}n$. Suppose $k\geq1$ and the upper bound holds for any $k'<k$.
Suppose to the contrary that $H$ is a $P_{\ell_0}^r\cup S_{\ell_1}^r\cup\cdots\cup S_{\ell_k}^r$-free linear $r$-graph on $n$ vertices with
\begin{eqnarray*}
|E(H)|> \Big(\frac{k}{r-1}+\frac{(2r-3)\ell}{2}\Big)(n-k)+\frac{k(k-1)}{r(r-1)}.
\end{eqnarray*}
To complete the proof we only need to show that $H$ contains a $P_{\ell_0}^r\cup k S_{\ell}^r$, where $\ell=\max\{\ell_0,\ell_1\}$. This contradicts the assumption that $H$ is $P_{\ell_0}^r\cup S_{\ell_1}^r\cup\cdots\cup S_{\ell_k}^r$-free.

Since $n$ is sufficiently large,
we have
\begin{eqnarray*}
|E(H)|&>& \Big(\frac{k}{r-1}+\frac{(2r-3)\ell}{2}\Big)(n-k)+\frac{k(k-1)}{r(r-1)} \\
&>& \Big(\frac{\ell-1}{r}+\frac{k-1}{r-1}\Big)(n-k+1)+\frac{(k-1)(k-2)}{r(r-1)} \\
&\geq& {\rm{ex}}_r^{\rm{lin}}(n,kS_\ell^r),
\end{eqnarray*}
which implies $H$ contains a $kS_\ell^r$ by Theorem 4.2. Take any one $S_\ell^r$, $H-V(S_\ell^r)$ is $P_{\ell_0}^r\cup S_{\ell_1}^r\cup\cdots\cup S_{\ell_{k-1}}^r$-free. By induction hypothesis,
\begin{eqnarray*}
|E(H-V(S_\ell^r))|\leq \Big(\frac{k-1}{r-1}+\frac{(2r-3)\ell}{2}\Big)(n-\ell(r-1)-k)+\frac{(k-1)(k-2)}{r(r-1)}.
\end{eqnarray*}
Thus, we obtain
\begin{eqnarray*}
|E(H)|-|E(H-V(S_\ell^r))|&>& \frac{n}{r-1}+\frac{2(k-1)}{r(r-1)}+\frac{\ell(r-1)(k-1)-k}{r-1}+\frac{(r-1)(2r-3)\ell^2}{2} \\
&>& \frac{n}{r-1},
\end{eqnarray*}
which implies each $S_\ell^r$ in $H$ contains a vertex with degree at least $\frac{n}{(r-1)(\ell(r-1)+1)}$.
Denote by $U$ the set of these $k$ vertices. Then $d_H(u)\geq\frac{n}{(r-1)(\ell(r-1)+1)}$ for any $u\in U$.
Since $H$ is linear, we have $|E(H[U])|\leq \frac{k(k-1)}{r(r-1)}$.

Let $\bar{U}=V(H)\backslash U$ and $E(U,\bar{U})$ be the set of edges that contain at least one vertex of $U$ and one vertex of $\bar{U}$. By Lemma 4.1, we have $|E(U,\bar{U})|\leq \frac{k(n-k)}{r-1}$.
Then we get
\begin{eqnarray*}
|E(H[\bar{U}])|=|E(H)|-|E(H[U])|-|E(U,\bar{U})|> \frac{(2r-3)\ell}{2}(n-k)\geq {\rm{ex}}_r^{\rm{lin}}(n-k,P_{\ell_0}^r),
\end{eqnarray*}
which implies $H[\bar{U}]$ contains a $P_{\ell_0}^r$ by Theorem 1.1. Note that the number of edges containing $u$ and at least one vertex from $U\backslash\{u\}$ is at most $k-1$ 
and $|V(P_{\ell_0}^r)|={\ell_0}(r-1)+1$. When $n$ is sufficiently large, we have
\begin{eqnarray*}
d_{H-V(P_{\ell_0}^r)}(u)&\geq& d_H(u)-(k-1)-{\ell_0}(r-1)-1 \\
&\geq& \frac{n}{(r-1)(\ell(r-1)+1)}-(k-1)-{\ell_0}(r-1)-1 \\
&>& k\ell(r-1)
\end{eqnarray*}
for any $u\in U$. Then we may find $k$ vertex-disjoint $S_\ell^r$ in $H-V(P_{\ell_0}^r)$ with center vertex in $U$. Thus, $H$ contains a $P_{\ell_0}^r\cup k S_{\ell}^r$.
This completes the proof.
\end{proof}

By Theorem 3.5, another upper bound for ${\rm{ex}}_r^{\rm{lin}}(n,P_{\ell_0}^r\cup S_{\ell_1}^r\cup\cdots\cup S_{\ell_k}^r)$ in terms of the linear Tur\'{a}n number of $P_\ell^r$ can be obtained.

\begin{thm}
Suppose integers $r\geq3$, $\ell_0\geq4$, $k\geq0$ and $c>0$. Let $\ell_1\geq\ell_2\geq\cdots\geq\ell_k$, $\ell=\max\{\ell_0,\ell_1\}$ and $r,\ell$ satisfy $(\ell-1)\big((\ell-1)(r-1)+1\big)\equiv0 ({\rm{mod}}\ r)$.
If $\ell$ is large and $n=c\big((\ell-1)(r-1)+1\big)$, then for sufficiently large $n$,
\begin{eqnarray*}
{\rm{ex}}_r^{\rm{lin}}(n,P_{\ell_0}^r\cup S_{\ell_1}^r\cup\cdots\cup S_{\ell_k}^r)\leq \frac{k(n-k)}{r-1}+\frac{k(k-1)}{r(r-1)}+{\rm{ex}}_r^{\rm{lin}}(n-k,P_\ell^r).
\end{eqnarray*}
\end{thm}
\begin{proof}[{\bf{Proof}}]
For $k=0$, the inequality holds. Suppose $k\geq1$ and the upper bound holds for any $k'<k$.
Suppose to the contrary that $H$ is a $P_{\ell_0}^r\cup S_{\ell_1}^r\cup\cdots\cup S_{\ell_k}^r$-free linear $r$-graph on $n$ vertices with
\begin{eqnarray*}
|E(H)|> \frac{k(n-k)}{r-1}+\frac{k(k-1)}{r(r-1)}+{\rm{ex}}_r^{\rm{lin}}(n-k,P_\ell^r).
\end{eqnarray*}
To complete the proof we only need to show that $H$ contains a $P_{\ell_0}^r\cup k S_{\ell}^r$, where $\ell=\max\{\ell_0,\ell_1\}$. This contradicts the assumption that $H$ is $P_{\ell_0}^r\cup S_{\ell_1}^r\cup\cdots\cup S_{\ell_k}^r$-free.

By Theorem 3.5, for sufficiently large $n$, we have
\begin{eqnarray*}
&{}&\frac{k(n-k)}{r-1}+\frac{k(k-1)}{r(r-1)}+{\rm{ex}}_r^{\rm{lin}}(n-k,P_\ell^r) \\
&\geq& \Big(\frac{k}{r-1}+\frac{\ell-1}{r}\Big)(n-k)+\frac{k(k-1)}{r(r-1)} \\
&>& \Big(\frac{\ell-1}{r}+\frac{k-1}{r-1}\Big)(n-k+1)+\frac{(k-1)(k-2)}{r(r-1)} \\
&\geq& {\rm{ex}}_r^{\rm{lin}}(n,kS_\ell^r).
\end{eqnarray*}
The rest part of the proof is similar to that of Theorem 1.2.
\end{proof}

\section{\normalsize Tur\'{a}n number for two kinds of linear star-path forests}
\ \ \ \
In this section, we mainly present the proofs of Theorems 1.3 and 1.4, which state the bounds for the Tur\'{a}n number of $P_\ell^r\cup kS_\ell^r$ and $k_1P_\ell^r\cup k_2S_\ell^r$ ($k_1\geq2$). Let us first provide three classical theorems that need to be used in the proof.

\begin{thm}[\cite{AA3}]
Let $r\geq3$ and $\ell\geq4$. For sufficiently large $n$,
\begin{eqnarray*}
{\rm{ex}}_r(n,P_\ell^r)={\rm{ex}}_r(n,C_\ell^r)=\binom{n}{r}-\binom{n-\lfloor\frac{\ell-1}{2}\rfloor}{r}+
\begin{cases}
0 &{\rm{if}}\ \ell\ {\rm{is}}\ {\rm{odd}}, \\
\binom{n-\lfloor\frac{\ell-1}{2}\rfloor-2}{r-2} &{\rm{if}}\ \ell\ {\rm{is}}\ {\rm{even}}.
\end{cases}
\end{eqnarray*}
\end{thm}

\begin{thm}[\cite{AA5}]
Let $r\geq3$, $\ell\geq1$ and $k\geq2$. For sufficiently large $n$,
\begin{eqnarray*}
{\rm{ex}}_r(n,kP_\ell^r)=\binom{n-1}{r-1}+\cdots+\binom{n-k\lfloor\frac{\ell+1}{2}\rfloor+1}{r-1}+
\begin{cases}
0 &{\rm{if}}\ \ell\ {\rm{is}}\ {\rm{odd}}, \\
\binom{n-k\lfloor\frac{\ell+1}{2}\rfloor-1}{r-2} &{\rm{if}}\ \ell\ {\rm{is}}\ {\rm{even}}.
\end{cases}
\end{eqnarray*}
\end{thm}

The problem of determining the Tur\'{a}n number for $S_\ell^r$ is still open. Duke and Erd\H{o}s \cite{AA4} gave an upper bound for ${\rm{ex}}_r(n,S_\ell^r)$ as follows.

\begin{thm}[\cite{AA4}]
Fix integers $\ell\geq2$ and $r\geq3$. Then there exists a constant $c(r)$ such that for sufficiently large $n$,
\begin{eqnarray*}
{\rm{ex}}_r(n,S_\ell^r)\leq c(r)\ell(\ell-1)n^{r-2}.
\end{eqnarray*}
\end{thm}
Note that when $n$ is sufficiently large and $\ell\,|\,n$, we have ${\rm{ex}}_2(n,P_\ell)={\rm{ex}}_2(n,S_\ell)=\frac{(\ell-1)n}{2}$ \cite{ZZ1,D4}. However, for any $r\geq3$ and $\ell\geq4$, we have ${\rm{ex}}_r(n,S_\ell^r)<{\rm{ex}}_r(n,P_\ell^r)$ by Theorems 5.1 and 5.3.

\begin{proof}[{\bf{Proof of Theorem 1.3.}}]
Let us first consider the upper bound. Suppose to the contrary that $H$ is an $r$-graph on $n$ vertices with
\begin{eqnarray*}
|E(H)|> \binom{n}{r}-\binom{n-k}{r}+{\rm{ex}}_r(n-k,P_\ell^r).
\end{eqnarray*}
We will show that $H$ contains a $P_\ell^r\cup kS_\ell^r$ by induction on $k$. If $k=0$, then $|E(H)|> {\rm{ex}}_r(n,P_\ell^r)$ and $H$ contains a $P_\ell^r$. Suppose $k\geq1$ and the statement holds for any $k'<k$.

{\bf{Case 1.}} $\Delta(H)<\frac{\binom{n}{r}-\binom{n-k}{r}}{k(\ell(r-1)+1)}$.

According to Theorem 5.1, we have
\begin{eqnarray*}
|E(H)|&>& \binom{n}{r}-\binom{n-k}{r}+{\rm{ex}}_r(n-k,P_\ell^r) \\
&>& \binom{n}{r}-\binom{n-k+1}{r}+{\rm{ex}}_r(n-k+1,P_\ell^r).
\end{eqnarray*}
By induction hypothesis, $H$ contains a $P_\ell^r\cup (k-1)S_\ell^r$. Note that ${\rm{ex}}_r(n,S_\ell^r)<{\rm{ex}}_r(n,P_\ell^r)$ by Theorems 5.1 and 5.3. Then
\begin{eqnarray*}
|E(H-V(P_\ell^r\cup (k-1)S_\ell^r))|&\geq& |E(H)|-k(\ell(r-1)+1)\Delta(H) \\
&>& {\rm{ex}}_r(n-k,P_\ell^r) \\
&>& {\rm{ex}}_r(n-k,S_\ell^r),
\end{eqnarray*}
which implies $H-V(P_\ell^r\cup (k-1)S_\ell^r)$ contains a $S_\ell^r$ by Theorems 5.3. So $H$ contains a $P_\ell^r\cup kS_\ell^r$.

{\bf{Case 2.}} $\Delta(H)\geq \frac{\binom{n}{r}-\binom{n-k}{r}}{k(\ell(r-1)+1)}$.

Suppose $u\in V(H)$ is a vertex with maximum degree. Since $d_H(u)=\Delta(H)\leq \binom{n-1}{r-1}$,
\begin{eqnarray*}
|E(H-\{u\})|&>& \binom{n}{r}-\binom{n-k}{r}+{\rm{ex}}_r(n-k,P_\ell^r)-\binom{n-1}{r-1} \\
&=& \binom{n-1}{r}-\binom{n-k}{r}+{\rm{ex}}_r(n-k,P_\ell^r).
\end{eqnarray*}
By induction hypothesis, $H-\{u\}$ contains a $P_\ell^r\cup (k-1)S_\ell^r$.

For any $v\in V(P_\ell^r\cup (k-1)S_\ell^r)$, the number of edges containing $u$ and $v$ is at most $\binom{n-2}{r-2}$. Note that $d_H(u)=\Delta(H)\geq \frac{\binom{n}{r}-\binom{n-k}{r}}{k(\ell(r-1)+1)}$ and $|V(P_\ell^r\cup (k-1)S_\ell^r)|=k(\ell(r-1)+1)$. When $n$ is sufficiently large,
\begin{eqnarray*}
d_{H-V\big(P_\ell^r\cup (k-1)S_\ell^r\big)}(u)&\geq& d_H(u)-k(\ell(r-1)+1)\binom{n-2}{r-2} \\
&>& (\ell(r-1)+1)\binom{n-2}{r-2}.
\end{eqnarray*}
Then we may find a $S_\ell^r$ that is vertex-disjoint from the $P_\ell^r\cup (k-1)S_\ell^r$. Thus, $H$ contains a $P_\ell^r\cup kS_\ell^r$.

We now continue with the lower bound.
Let $U$ be the set of $k$ vertices and $G$ be a $\{P_\ell^r, S_\ell^r\}$-free $r$-graph with $n-k$ vertices and ${\rm{ex}}_r(n-k,\{P_\ell^r, S_\ell^r\})$ edges. Let $H^*$ be an $r$-graph on $n$ vertices obtained from $G$ by adding all vertices in $U$ and all edges that is incident to $U$. Since $G$ is $\{P_\ell^r, S_\ell^r\}$-free, any $P_\ell^r$ or $S_\ell^r$ in $H^*$ contains at least one vertex in $U$. So $P_\ell^r\cup kS_\ell^r$ in $H^*$ contains at least $k+1$ distinct vertices in $U$. Since $|U|=k$, $H^*$ is $P_\ell^r\cup kS_\ell^r$-free.

By the definition of $H^*$, we have
\begin{eqnarray*}
|E(H^*)|=\binom{n}{r}-\binom{n-k}{r}+{\rm{ex}}_r(n-k,\{P_\ell^r, S_\ell^r\}).
\end{eqnarray*}
This completes the proof.
\end{proof}

\begin{proof}[{\bf{Proof of Theorem 1.4.}}]
Let us first consider the upper bound. Suppose to the contrary that $H$ is an $r$-graph on $n$ vertices with
\begin{eqnarray*}
|E(H)|> \binom{n}{r}-\binom{n-k_2}{r}+{\rm{ex}}_r(n-k_2,k_1P_\ell^r).
\end{eqnarray*}
We will show that $H$ contains a $k_1P_\ell^r\cup k_2S_\ell^r$ by induction on $k_2$. If $k_2=0$, then $|E(H)|> {\rm{ex}}_r(n,k_1P_\ell^r)$ and $H$ contains a $k_1P_\ell^r$. Suppose $k_2\geq1$ and the statement holds for any $k'<k_2$.

{\bf{Case 1.}} $\Delta(H)<\frac{\binom{n}{r}-\binom{n-k_2}{r}}{(k_1+k_2-1)(\ell(r-1)+1)}$.

According to Theorem 5.2, we have
\begin{eqnarray*}
|E(H)|&>& \binom{n}{r}-\binom{n-k_2}{r}+{\rm{ex}}_r(n-k_2,k_1P_\ell^r) \\
&>& \binom{n}{r}-\binom{n-k_2+1}{r}+{\rm{ex}}_r(n-k_2+1,k_1P_\ell^r).
\end{eqnarray*}
By induction hypothesis, $H$ contains a $k_1P_\ell^r\cup (k_2-1)S_\ell^r$. Note that ${\rm{ex}}_r(n,S_\ell^r)<{\rm{ex}}_r(n,P_\ell^r)<{\rm{ex}}_r(n,k_1P_\ell^r)$ by Theorems 5.1, 5.2 and 5.3. Then
\begin{eqnarray*}
|E(H-V(k_1P_\ell^r\cup (k_2-1)S_\ell^r))|&\geq& |E(H)|-(k_1+k_2-1)(\ell(r-1)+1)\Delta(H) \\
&>& {\rm{ex}}_r(n-k_2,k_1P_\ell^r) \\
&>& {\rm{ex}}_r(n-k_2,S_\ell^r),
\end{eqnarray*}
which implies $H-V(k_1P_\ell^r\cup (k_2-1)S_\ell^r)$ contains a $S_\ell^r$ by Theorems 5.3. So $H$ contains a $k_1P_\ell^r\cup k_2S_\ell^r$.

{\bf{Case 2.}} $\Delta(H)\geq \frac{\binom{n}{r}-\binom{n-k_2}{r}}{(k_1+k_2-1)(\ell(r-1)+1)}$.

Suppose $u\in V(H)$ be a vertex with maximum degree. Since $d_H(u)=\Delta(H)\leq \binom{n-1}{r-1}$,
\begin{eqnarray*}
|E(H-\{u\})|&>& \binom{n}{r}-\binom{n-k_2}{r}+{\rm{ex}}_r(n-k_2,k_1P_\ell^r)-\binom{n-1}{r-1} \\
&=& \binom{n-1}{r}-\binom{n-k_2}{r}+{\rm{ex}}_r(n-k_2,k_1P_\ell^r).
\end{eqnarray*}
By induction hypothesis, $H-\{u\}$ contains a $k_1P_\ell^r\cup (k_2-1)S_\ell^r$.

For any $v\in V(k_1P_\ell^r\cup (k_2-1)S_\ell^r)$, the number of edges containing $u$ and $v$ is at most $\binom{n-2}{r-2}$. Note that $d_H(u)=\Delta(H)\geq \frac{\binom{n}{r}-\binom{n-k_2}{r}}{(k_1+k_2-1)(\ell(r-1)+1)}$ and $|V(k_1P_\ell^r\cup (k_2-1)S_\ell^r)|=(k_1+k_2-1)(\ell(r-1)+1)$. When $n$ is sufficiently large,
\begin{eqnarray*}
d_{H-V\big(k_1P_\ell^r\cup (k_2-1)S_\ell^r\big)}(u)&\geq& d_H(u)-(k_1+k_2-1)(\ell(r-1)+1)\binom{n-2}{r-2} \\
&>& (\ell(r-1)+1)\binom{n-2}{r-2}.
\end{eqnarray*}
Then we may find a $S_\ell^r$ that is vertex-disjoint from the $k_1P_\ell^r\cup (k_2-1)S_\ell^r$. Thus, $H$ contains a $k_1P_\ell^r\cup k_2S_\ell^r$.

We now continue with the lower bound.
Let $U$ be the set of $k_1+k_2-1$ vertices and $G$ be a $\{P_\ell^r, S_\ell^r\}$-free $r$-graph with $n-k_1-k_2+1$ vertices and ${\rm{ex}}_r(n-k_1-k_2+1,\{P_\ell^r, S_\ell^r\})$ edges. Let $H^*$ be an $r$-graph on $n$ vertices obtained from $G$ by adding all vertices in $U$ and all edges that is incident to $U$. Since $G$ is $\{P_\ell^r, S_\ell^r\}$-free, any $P_\ell^r$ or $S_\ell^r$ in $H^*$ contains at least one vertex in $U$. So $k_1P_\ell^r\cup k_2S_\ell^r$ in $H^*$ contains at least $k_1+k_2$ distinct vertices in $U$. Since $|U|=k_1+k_2-1$, $H^*$ is $k_1P_\ell^r\cup k_2S_\ell^r$-free.

By the definition of $H^*$, we have
\begin{eqnarray*}
|E(H^*)|=\binom{n}{r}-\binom{n-k_1-k_2+1}{r}+{\rm{ex}}_r(n-k_1-k_2+1,\{P_\ell^r, S_\ell^r\}).
\end{eqnarray*}
This completes the proof.
\end{proof}

\ \ \ \

\noindent{\bf{Declaration of interest}}

The authors declare no known conflicts of interest.

\smallskip

\smallskip

\noindent{\bf{Acknowledgements}}

The authors would like to express their sincere gratitude to the editor and the referees for a very careful reading of the paper and for all their insightful comments and valuable suggestions, which led to a number of improvements in this paper.

\end{document}